\documentclass[a4paper]{amsart}

\usepackage{amsmath,amsthm,amssymb}
\usepackage{hyperref}
\usepackage{todonotes}
\usepackage[all]{xy}

\newtheorem{theorem}{Theorem}[section]
\newtheorem{proposition}[theorem]{Proposition}
\newtheorem{lemma}[theorem]{Lemma}
\newtheorem{corollary}[theorem]{Corollary}

\theoremstyle{definition}

\newtheorem{definition}[theorem]{Definition}

\newcommand{\nat}{\mathbb{N}}
\newcommand{\ZZ}{\mathbb{Z}}
\newcommand{\gsets}[1]{#1\mathbf{-Set}}
\newcommand{\set}{\mathbf{Set}}
\newcommand{\baire}{\nat^\nat}
\newcommand{\tope}{\mathcal{E}}
\newcommand{\topf}{\mathcal{F}}
\newcommand{\iso}{\operatorname{Iso}}
\newcommand{\dihedral}{D_\infty}

\begin{document}
\title{On Dividing by Two in Constructive Mathematics}

\author{Andrew W Swan}
\address{Andrew Swan\\
Institute for Logic, Language and Computation\\
University of Amsterdam\\
Science Park 107\\
1098 XG Amsterdam\\
Netherlands}
\email{wakelin.swan@gmail.com}

\begin{abstract}
  A classic result due to Bernstein states that in set theory with
  classical logic, but without the axiom of choice, for all sets $X$
  and $Y$, if $X \times 2 \cong Y \times 2$ then also $X \cong Y$. We
  show that this cannot be done in constructive mathematics by giving
  some examples of toposes where it fails.
\end{abstract}

\maketitle

\section{Introduction}

In classical set theory with the axiom of choice it is trivial to show
that given two sets $X$ and $Y$, if there exists a bijection
$X \times 2 \cong Y \times 2$, then there exists also a bijection $X
\cong Y$. We refer to this statement as ``dividing by two.''

Surprisingly, it is possible to divide by two in even in absence of
the axiom of choice, by explicitly defining a bijection between $X$
and $Y$, given a bijection between $X \times 2$ and $Y \times 2$. This
was first proved by Bernstein in his thesis in 1905
\cite{Bernstein1905}. Later, in 1922 in \cite{Sierpinski1922},
Sierpi\'nski gave a simplified proof of the same result. For a more
recent exposition see the work of Conway and Doyle \cite{conwaydoyle},
who gave a proof of both this result and the more difficult problem of
dividing by three.

More recently still, this construction has seen some attention on
various forums online. In particular, the Mathematics Stack Exchange
user, Hanno, raised the question that we will answer in this paper:
whether division by two can be carried out in constructive
mathematics \cite{msequestion}.

Here, by constructive mathematics, we mean mathematics carried out
without the use of excluded middle, in the style of, for instance,
Bishop and Bridges in \cite{bishopbridges}. Toposes are
categories analogous to the category of sets, wherein mathematical
statements and constructive proofs can be interpreted. See
e.g. \cite{moerdijkmaclane} for an introduction to topos theory.

The main ideas of the proof
(including sections \ref{sec:basic}, \ref{sec:equivariance},
\ref{sec:continuity}, \ref{sec:nonexistence} and most of section
\ref{sec:maketopos}) can be understood with little to no prior
knowledge of constructive mathematics and topos theory. Some basic
knowledge of group theory will be useful, however.

\section{Outline of the Proof}
\label{sec:basic}

The key part of our proof that dividing by two is impossible in
constructive mathematics is in fact the same idea used in proving that
it can be done in $\mathbf{ZF}$.

Suppose that we are given two sets $X$ and $Y$ together with a
bijection $X \times 2 \cong Y \times 2$. Following Sierpi\'nski
\cite{Sierpinski1922} we will think of the bijection as a permutation
$\theta$ of $(X \times 2) + (Y \times 2)$ of order 2. Namely, $\theta$
takes each element of $X \times 2$ to the corresponding element of
$Y \times 2$ and vice versa.

Note that $(X \times 2) + (Y \times 2) \cong (X + Y) \times 2$. We
will write $Z$ for $X + Y$, and so think of $\theta$ as a permutation
of $Z \times 2$.

There is another permutation of $Z \times 2$ that sends
$(z, i) \in Z \times 2$ to $(z, 1 - i)$. We denote this permutation
$\phi$.

Now we note that given any element $z$ of $Z \times 2$, we can define
a sequence $\chi_z \colon \ZZ \to 2$, as follows. Writing $\pi_1$ for
the projection $Z \times 2 \to 2$, we define $\chi_z(n)$ to be
$\pi_1((\phi \theta)^n \cdot z)$. The construction of the bijection $X
\cong Y$ proceeds by dividing into cases depending on the properties
of $\chi$. We refer to \cite{Bernstein1905}, \cite{Sierpinski1922} or
\cite{conwaydoyle} for details.

The main idea for our proof is that in general we always have the
computational information contained in $\chi_x$ available anyway. We
can therefore clarify the situation by making the dependence on the
signature explicit. We will take $X$ to be the even numbers, $2 \ZZ$,
and $Y$ to be the odd numbers $2 \ZZ + 1$. Instead of starting with a
bijection $\theta$ and then defining $\chi$, we will start off with an
arbitrary sequence $\chi \colon \ZZ \to 2$, and then define a
bijection $\theta_\chi$, as follows.
\begin{definition}
  Suppose we are given $\chi \colon \ZZ \rightarrow 2$. We define a
  bijection $\theta_\chi$ from $2 \ZZ \times 2$ to
  $(2 \ZZ + 1) \times 2$  as follows.
  \begin{equation*}
    \theta_\chi(n, i) :=
    \begin{cases}
      (n + 1, 1 - \chi(n + 1)) & i = \chi(n) \\
      (n - 1, \chi(n - 1)) & i \neq \chi(n) \\
    \end{cases}
  \end{equation*}  
\end{definition}

To see that $\theta_\chi$ is a bijection, note that we can
use exactly the same definition for its inverse (although in that case
the domain consists of odd numbers rather than even numbers).

Dividing by two gives us a way to transform each bijection
$\theta_\chi$ between $2 \ZZ \times 2$ and $(2 \ZZ + 1) \times 2$ into
a bijection between $2 \ZZ$ and $2 \ZZ + 1$.

Obviously there are always bijections between $2\ZZ$ and $2 \ZZ + 1$,
for instance by adding $1$. However, our proof will rest on the fact
that $\theta_\chi$ depends on $\chi$ equivariantly and continuously,
in a sense that we will define. We will see that these cannot both
hold for families of bijections between $2 \ZZ$ and $2 \ZZ + 1$.

\section{Equivariance}
\label{sec:equivariance}

Let $\dihedral$ be the infinite dihedral group. There are a few
different ways of defining this group. We will use the finite
presentation below.
\begin{equation*}
  \dihedral := \langle t, r \;|\; r^2 = 1, r t r = t^{-1}
  \rangle
\end{equation*}
It's helpful to note that $\dihedral$ is also the group of isometries of
$\ZZ$, with $t$ translation by $1$, and $r$ reflection in the origin,
say. For this reason, we will refer to elements of the form $t^n$ as
\emph{translations}, and elements of the form $r t^n$ as
\emph{reflections}.

In fact, we don't use the usual action of $\dihedral$ on $\ZZ$, but
instead define the action of translation
to be translation by $2$, so as to preserve odd and
even numbers. Explicitly, we define $t \cdot n$ to be $n + 2$ and we
define $r \cdot n$ to be $- n$.

We define an action on $2^\ZZ$ as follows. Suppose we are given an
element $\chi$ of $2^\ZZ$. We define $t \cdot \chi$ to be
$\lambda n.\chi(n - 2)$ and we define $r \cdot \chi$ to be
$\lambda n.\,1 -\chi(-n)$.
\begin{proposition}
  The above specifies a well defined action of $\dihedral$ on $\ZZ$
  which restricts to $2 \ZZ$ and $2 \ZZ + 1$, and a well defined
  action on $2^\ZZ$.
\end{proposition}

\begin{proof}
  To show these are well defined actions it suffices to check that they
  respect the equations in the finite presentation of $\dihedral$, which is
  straightforward.
\end{proof}

We also consider the trivial action on $2$, which then defines an
action on $2\ZZ \times 2$ and $(2\ZZ + 1) \times 2$.

\begin{lemma}
  Let $g \in \dihedral$ and $\chi \in 2^\ZZ$. Then $\theta_{g \cdot \chi}(g \cdot n) =
  g \cdot \theta_\chi(n)$.
\end{lemma}

\begin{proof}
  It suffices to check this for the generators $t$ and $r$. Both are
  straightforward.
\end{proof}

\begin{proposition}
  \begin{enumerate}
  \item \label{li:naivenotequiv}
    The na\"ive bijections between $2 \ZZ$ and $2 \ZZ + 1$ defined
    by always adding $1$
    (or alternatively always subtracting $1$) are not equivariant in
    $2^\ZZ$.
  \item \label{li:bernsteinisequiv} There is an equivariant family of
    bijections between $2 \ZZ$ and $2 \ZZ + 1$ (assuming the axiom of
    excluded middle).
  \end{enumerate}
\end{proposition}

\begin{proof}
  We leave a direct proof of \ref{li:naivenotequiv} as an exercise for
  the reader, although it will also follow from the arguments we will
  use in section \ref{sec:nonexistence}.

  For \ref{li:bernsteinisequiv}, we leave it as an exercise for the
  reader to give a direct proof based on Bernstein's argument, but we
  also give an abstract proof based on the ideas that we will see in
  section \ref{sec:maketopos}: note that if excluded middle holds,
  then it also holds internally in the topos of $\dihedral$-sets,
  i.e. $\gsets{\dihedral}$ is a boolean topos. We can therefore carry
  out Bernstein's argument in the internal logic of $\gsets{\dihedral}$
  to get an equivariant bijection.
\end{proof}

\section{Continuity}
\label{sec:continuity}

Recall that Baire space is defined to be the topological space on the
set $\nat^\nat$ with the product topology on $\nat$ copies
of $\nat$ with the discrete topology.
It's often useful (in both constructive and classical
mathematics) to give an explicit definition of continuity as follows.

\begin{proposition}
  A function $F \colon \baire \to \nat$ is continuous iff for
  all $\alpha \in \baire$ there exists $N \in \nat$ such that for all
  $\beta \in \baire$ if $\alpha(n) = \beta(n)$ for all $n < N$, then
  $F(\alpha) = F(\beta)$.
\end{proposition}

\begin{proposition}
  The bijections $(\theta_\chi)_{\chi \in 2^\ZZ}$ that we defined in
  section \ref{sec:basic} are continuous when viewed as a single
  function from $2^\ZZ \times 2 \ZZ$ to $2^\ZZ \times (2 \ZZ + 1)$.

  More explicitly, for every $\chi$ and every $n \in \ZZ$, there
  exists $N \in \nat$ such that for all $\chi' \in 2^\ZZ$, if
  $\chi'(m) = \chi(m)$ for all $m$ with $|m| < N$, then
  $\theta_{\chi}(n) = \theta_{\chi'}(n)$.
\end{proposition}

\begin{proof}
  This is clear from the definition.
\end{proof}

We say that $\theta_\chi$ is \emph{continuous in $2^\ZZ$}.

We make the following observation. A direct proof is left as an
exercise for the reader, although it will also follow as a corollary
from the arguments in the next section.
\begin{proposition}
  The bijections between $2 \ZZ$ and $2 \ZZ + 1$ resulting from Bernstein's
  construction are not continuous in $2^\ZZ$.
\end{proposition}

\section{The Non-Existence of a Continuous and Equivariant Family of
  Bijections}
\label{sec:nonexistence}

We aim towards a proof that there is no continuous and equivariant
family of bijections between $2 \ZZ$ and $2 \ZZ + 1$ in the sense that
we defined in previous sections (this will be theorem
\ref{thm:noequivhomeo}). It will be important for future sections that
this proof is entirely constructive (and so valid in the internal
logic of a topos).

Our first observation is that we do not need to consider all
elements of $2^\ZZ$, but only those $\chi$ that are decreasing. We use
the following notation for some of these elements. We write $- \infty$
for the sequence which is constantly $0$, $\infty$ for the sequence
which is constantly equal to $1$, and given $n \in \ZZ$, we write
$\underline{n}$ for the sequence defined as below.
\begin{equation*}
  \underline{n}(i) :=
  \begin{cases}
    1 & i < n \\
    0 & i \geq n
  \end{cases}
\end{equation*}
We will write the set of decreasing sequences as $\ZZ_\infty$.

Although we don't formally need the following two observations, they
help illustrate the motivation for this definition and notation.
\begin{enumerate}
\item For any $m < n$ in $\ZZ$, in the pointwise ordering on $2^\ZZ$,
  $- \infty < \underline{m} < \underline{n} < \infty$.
\item In classical logic, every decreasing binary sequence on $\ZZ$ is
  either of the form $- \infty$, $\infty$ or $\underline{n}$ for some
  $n \in \ZZ$. In fact this is equivalent to one of Brouwer's
  omniscience principles, the \emph{limited principle of omniscience
    (LPO)}, which states that for every binary sequence $\alpha \colon
  \nat \to 2$, either $\alpha(n) = 0$ for every $n \in \nat$, or
  $\alpha(n) = 1$ for some $n \in \nat$.
\end{enumerate}

Throughout this section, we assume that we are given a family of
bijections $\varphi_\chi \colon 2 \ZZ \to 2 \ZZ + 1$ indexed by
elements of $\ZZ_\infty$ that are both continuous and equivariant in
$\ZZ_\infty$, in the sense that we defined in the previous
sections. Clearly if we are given a continuous equivariant family of
bijections indexed by all elements of $2^\ZZ$, we could just
restrict to get a continuous equivariant family of bijections indexed
over $\ZZ_\infty$.

We will often view the family of bijections $\varphi$ as a function
$\ZZ_\infty \times 2 \ZZ \,\to\, \ZZ_\infty \times (2 \ZZ + 1)$ that forms
part of the following commutative triangle.

\begin{equation*}
  \xymatrix{ \ZZ_\infty \times 2 \ZZ \ar[rr]^{\varphi} \ar[dr]_{\pi_0} & &
    \ZZ_\infty \times (2 \ZZ + 1) \ar[dl]^{\pi_0} \\
    & \ZZ_\infty &}
\end{equation*}




\begin{lemma}
  \label{lem:eventuallylinear}
  There exists $k \in \ZZ$ and $N > 0$ such that for all $n > N$,
  $\varphi(\underline{0}, n) = (\underline{0}, n + k)$, and for all
  $n < -N$, $\varphi(\underline{0}, n) = (\underline{0}, n - k)$.
\end{lemma}

\begin{proof}
  Let $k$ be such that $\varphi(-\infty, 0) = (-\infty, k)$.
  
  Since $\varphi$ is a continuous function,
  $\pi_1 \circ \varphi$ is also continuous as a function
  $\ZZ_\infty \times 2 \ZZ \rightarrow \ZZ$.  Hence there exists
  $N > 0$ such that for all $\chi$, if $\chi(i) = 0$ for $|i| < N$
  then we have
  $\pi_1 (\varphi(\chi, 0)) = \pi_1(\varphi(-\infty, 0))$.
  It clearly follows that $\varphi(-\underline{n}, 0) =
  \varphi(-\infty, 0)$ for $n \leq -N$.

  Then we have the following for every even number $2n$ with $2n
  > N$.
  \begin{align*}
    \pi_1(\varphi(\underline{0}, 2 n)) &= \pi_1(\varphi( t^n \cdot
                                         (\underline{-2 n}, 0)))
    \\
                                     &= t^n \cdot \pi_1 (\varphi(
                                       \underline{-2n}, 0)) \\
                                     &= t^n \cdot \pi_1(\varphi( -\infty, 0)) \\
                                     &= t^n \cdot \pi_1(-\infty, k) \\
                                     &= 2 n + k
  \end{align*}
  For $2n < -N$ we have $2n + 1 \leq -N$, and hence,
  \begin{align*}
    \pi_1(\varphi(\underline{0}, 2n))
    &= \pi_1(\varphi(t^n \cdot (\underline{-2n}, 0))) \\
    &= \pi_1(\varphi(t^n r \cdot (\underline{2n + 1}, 0))) \\
    &= t^{n} r \cdot \pi_1(\varphi(\underline{2n + 1}, 0)) \\
    &= t^n r \cdot \pi_1(\varphi(-\infty, 0)) \\
    &= t^n r \cdot \pi_1(-\infty, k) \\
    &= 2n - k
  \end{align*}
\end{proof}

\begin{theorem}
  \label{thm:noequivhomeo}
  There is no family of bijections between $2 \ZZ$ and $2 \ZZ + 1$
  that is continuous and equivariant in $\ZZ_\infty$.
\end{theorem}

\begin{proof}
  In lemma \ref{lem:eventuallylinear} we showed that there
  exists $N > 0$ (which is even without loss of generality) and
  $k \in \ZZ$ (which is necessarily odd) such that for $n > 0$,
  $\varphi_{\underline{0}}(n) = n + k$ and for $n < -N$, $\varphi_{\underline{0}}(n) = n -
  k$. Furthermore, without loss of generality $N > |k|$.

  We deduce that the image of the set
  $((-\infty, -N -2] \cup [N + 2, \infty))\cap 2\ZZ$ under
  $\varphi_{\underline{0}}$ must be
  $((-\infty, N - k - 2] \cup [N + k + 2, \infty)) \cap (2 \ZZ + 1)$.
  Since $\varphi_{\underline{0}}$
  is a bijection, it follows that the image of the set $[-N, N] \cap 2\ZZ$ is
  $[-N - k, N + k] \cap (2\ZZ + 1)$.
  However, the cardinality of $[-N, N] \cap 2\ZZ$
  is odd, because for each even number $n$ with $0 < n \leq N$ it
  contains $n$ and $-n$, and it also contains $0$. On the other hand
  the cardinality of $[-N - k, N + k] \cap 2\ZZ + 1$ is even because
  we can still pair up each element with its negation, but it does not
  contain $0$.
\end{proof}

\section{Construction of the Topos}
\label{sec:maketopos}

We first consider the topos $\gsets{\dihedral}$ of sets with $\dihedral$-action where
$\dihedral$ is the infinite dihedral group. Recall that an object is a set $X$
together with an action of $\dihedral$ on $X$, and a morphism is a function
that preserves the action.

Our first steps will look a little strange to readers unfamiliar with
constructive mathematics.

\begin{lemma}
  \label{lem:allfuncontretract}
  Suppose that every function from $\baire$ to $\nat$ is continuous
  (this is sometimes referred to as \emph{Brouwer's continuity
    principle} or just \emph{Brouwer's principle}). Then the same
  is true for every retract of $\baire$.
\end{lemma}

\begin{proof}
  Suppose that $S$ is a retract of $\baire$. Then the
  inclusion $\iota \colon S \to \baire$ is continuous, and by
  definition there is a continuous map $p \colon \baire \to S$ such
  that $p \circ \iota = 1_S$.

  Let $f$ be any continuous function from $S$ to $\nat$. Note that
  $f \circ p$ is a function from $\baire$ to $\nat$, and
  so continuous. But then $f = f \circ p \circ \iota$,
  and so $f$ must also be continuous.
\end{proof}

\begin{lemma}
  \label{lem:exfromallfuncont}
  Suppose that every function $\baire \to \nat$ is
  continuous. Then the slice category $\gsets{\dihedral}/\ZZ_\infty$ contains
  two objects $X$ and $Y$ such that $X \times 2 \cong Y \times 2$ but
  there is no isomorphism between $X$ and $Y$.
\end{lemma}

\begin{proof}
  We consider the example from theorem \ref{thm:noequivhomeo}. Note
  that the category of equivariant families of maps over $\ZZ_\infty$
  is equivalent to the slice category $\gsets{\dihedral}/\ZZ_\infty$.

  The space $\ZZ_\infty \times \ZZ$ is evidently a retract of
  $\baire$, and so every function to $\ZZ$ continuous by lemma
  \ref{lem:allfuncontretract}. Hence every function
  $\ZZ_\infty \times 2\ZZ \to \ZZ_\infty \times (2\ZZ + 1)$ is
  continuous, and in particular any family of bijections $\varphi$. We
  can now apply theorem \ref{thm:noequivhomeo}.
\end{proof}

\begin{theorem}
  \label{thm:main}
  There is topos $\topf$ containing objects $X$ and $Y$ such
  that $X \times 2 \cong Y \times 2$ but $X \not \cong Y$.
\end{theorem}

\begin{proof}
  Let $\tope$ be any topos that satisfies Brouwer's continuity axiom
  that all functions $\baire \to \nat$ are continuous. This includes
  a couple of well known toposes in realizability, the effective topos
  and the function realizability topos (see Proposition 3.1.6 and
  Proposition 4.3.4 respectively in \cite{vanoosten}). As shown by Van
  der Hoeven and Moerdijk in \cite{vdhoevenmoerdijk}, it is also
  possible to construct a topos of sheaves with this property.

  A well known result in topos theory is that one can construct a
  category of internal $G$-sets from an internal group $G$ and that
  this category is again a topos. See \cite[Section
  V.6]{moerdijkmaclane} for more details.

  Note that we can construct the infinite dihedral group $\dihedral$ in the
  internal logic of $\tope$, to obtain an internal group $(\dihedral)_\tope$ in
  $\tope$. We then apply the construction of internal $G$-sets to
  $(\dihedral)_\tope$ and refer to the resulting topos as $(\gsets{\dihedral})_\tope$.

  Next we define $\ZZ_\infty$ and its action internally in $\tope$
  to obtain an object of $(\gsets{\dihedral})_\tope$ and take our topos
  $\topf$ to be the slice category $(\gsets{\dihedral})_\tope /
  \ZZ_\infty$. It is again well known that every slice category of a
  topos is again a topos (see \cite[Section IV.7]{moerdijkmaclane}).

  Finally, we note that our proof of lemma \ref{lem:exfromallfuncont}
  is entirely constructive, so we can carry it out in the internal
  logic of $\tope$. However, we can now remove the assumption that all
  functions $\baire \to \nat$ are continuous, since we chose $\tope$
  so that it holds in the internal logic.

  This then gives us two objects $X$ and $Y$ in $(\gsets{\dihedral})_\tope /
  \ZZ_\infty$ such that $X \times 2 \cong Y \times 2$ but $X \not
  \cong Y$, as required.
\end{proof}

\section{Strengthenings of the Main Theorem}

We now consider two slightly stronger versions of the main
theorem. There are two issue that we address.

The first is that one might expect that it becomes possible to
construct the bijection $X \cong Y$ if we add the extra requirement
that $X$ and $Y$ have decidable equality. We check that in fact
decidable equality of $X$ and $Y$ already holds in the topos we have
constructed, and so it does not help.
\begin{proposition}
  \label{prop:deceqmain}
  The objects $X$ and $Y$ considered in theorem \ref{thm:main} have
  decidable equality.
\end{proposition}

\begin{proof}  
  Note that since $\ZZ$ has decidable equality (provably in
  constructive mathematics), we can show internally in $\tope$ that
  there is a decision function
  $2\ZZ \times_{\ZZ_\infty} 2\ZZ \rightarrow \ZZ_\infty \times 2$ and
  similarly for $Y$. Since equality is preserved by the action of a
  group, the decision functions are equivariant, and so witness the
  decidable equality of $X$ and $Y$ in $(\gsets{\dihedral})_\tope/\ZZ_\infty$.
\end{proof}

\begin{corollary}
  There is topos $\topf$ containing objects $X$ and $Y$ with
  decidable equality such that $X \times 2 \cong Y \times 2$ but
  $X \not \cong Y$.
\end{corollary}

The next issue is a little subtle. Essentially, it might happen that
even though there is no isomorphism $X \cong Y$ in the topos, the
statement ``$X$ and $Y$ are isomorphic'' still holds in the internal
logic of the topos. One way of looking at this is that we can
construct the collection of bijections between $X$ and $Y$ in the
internal logic of the topos, to give an object $\iso(X, Y)$. External
isomorphisms then correspond to global sections of $\iso(X, Y)$,
i.e. to maps $1 \rightarrow \iso(X, Y)$. Meanwhile, the internal truth
of the statement ``$X$ and $Y$ are isomorphic'' corresponds to the
unique map $\iso(X, Y) \rightarrow 1$ being an epimorphism, which is
weaker. In fact in our topos $\topf$, exactly this happens. In this case
$\iso(X, Y)$ consists of all bijections from $X$ to $Y$, where the
isomorphisms are precisely the equivariant elements.

To deal with this, we will show in the next lemma how to construct a
new topos where $\iso(X, Y) \rightarrow 1$ is not an epimorphism.
\begin{lemma}
  \label{lem:noisointernal}
  Suppose that we are given a topos $\topf$ with objects $X$ and $Y$,
  and isomorphism $X \times 2 \cong Y \times 2$ such that there exists
  no (external) isomorphism $X \cong Y$. Then there is a topos
  $\mathcal{F}'$ with objects $X'$ and $Y'$ and an isomorphism
  $X' \cong Y'$ such that the internal logic of $\topf'$ does not
  satisfy the statement ``there exists a bijection from $X$ to $Y$''.
\end{lemma}

\begin{proof}
  We take $\topf'$ to be the Sierpi\'{n}ski cone of $\topf$, which we recall
  is defined to be the comma category $(\set \downarrow \Gamma)$,
  where $\Gamma \colon \topf \rightarrow \set$ is the global sections
  functor, $\topf(1, -)$ (see
  e.g. \cite[Example A2.1.12]{theelephant}).

  We define $X'$ to be the unique map $0 \to \Gamma X$, and $Y'$ to be
  the unique map $0 \to \Gamma Y$. In $(\set \downarrow \Gamma)$,
  limits and colimits are computed levelwise. Hence $2$ is given
  by the canonical map $2 \to \Gamma 2$, and $X' \times 2$ is the
  canonical map
  $0 \times 2 \to \Gamma (X \times 2)$. Since $0 \times 2 \cong 0$,
  this means $X' \times 2$ is the unique map $0 \to \Gamma(X \times
  2)$. Similarly, $Y' \times 2$ must be the unique map $0 \to \Gamma(Y
  \times 2)$. We can now clearly see that the isomorphism $X \times 2
  \cong Y \times 2$ lifts to an isomorphism $X' \times 2 \cong Y'
  \times 2$.

  Now $\iso(X', Y')$ has to be of the form $\iso(X', Y')_0 \to \Gamma
  (\iso(X', Y')_1)$. Since the projection from $(\set \downarrow
  \Gamma)$ to $\topf$ is logical, we have that $\iso(X', Y')_1 \cong
  \iso(X, Y)$. So $\iso(X', Y')$ is of the form $\iso(X', Y')_0 \to \Gamma
  (\iso(X, Y))$.

  By assumption there are no isomorphisms from $X$ to $Y$. Hence
  $\iso(X, Y)$ has no global sections, which precisely says that
  $\Gamma(\iso(X, Y))$ is the empty set. Hence $\iso(X', Y')_0$ must
  also be empty.

  Finally, we note that the statement ``there exists an isomorphism
  from $X'$ to $Y'$'' holds in the internal logic if and only if
  the unique map $\iso(X', Y') \to 1$ is an epimorphism. This is the
  case precisely when both of the maps $\iso(X', Y')_0 \to 1$ and
  $\iso(X', Y')_1 \to 1$ are epimorphisms. Although the latter might
  be epi (which is exactly why we need this lemma), the
  former is certainly not, since it is a function from the empty
  set to $1$.
\end{proof}

\begin{lemma}
  \label{lem:freydcoverdeceq}
  Suppose that $X$ and $Y$ have decidable equality in a topos $\topf$. Then
  $X'$ and $Y'$ have decidable equality in the topos $\mathcal{F}'$
  constructed in lemma \ref{lem:noisointernal}.
\end{lemma}

\begin{proof}
  It is straightforward to show that the decision morphism $X \times X
  \to 2$ lifts to a morphism $X' \times X' \to 2$ witnessing that $X'$
  has decidable equality, and similarly for $Y'$.
\end{proof}

\begin{corollary}
  \label{cor:toposweverything}
  There is a topos $\mathcal{F}'$ containing objects $X$ and $Y$, and
  an isomorphism $X \times 2 \cong Y \times 2$ such that the statement
  ``there exists a bijection from $X$ to $Y$'' does not hold in the
  internal logic of $\mathcal{F}'$.

  Furthermore $X$ and $Y$ have decidable equality.
\end{corollary}

\begin{proof}
  Apply lemma \ref{lem:noisointernal} to theorem \ref{thm:main}.
  For decidable equality, apply lemma \ref{lem:freydcoverdeceq} and
  proposition \ref{prop:deceqmain}.
\end{proof}

\begin{corollary}
  It is not provable in constructive mathematics that if
  $X \times 2 \cong Y \times 2$ then $X \cong Y$, even if we require
  $X$ and $Y$ to have decidable equality.
\end{corollary}

\begin{proof}
  If there was a constructive proof of this statement, then it would
  hold in the internal logic of any topos, contradicting
  corollary \ref{cor:toposweverything}.
\end{proof}

\bibliographystyle{abbrv}
\bibliography{dividebytwo}

\end{document}